\documentclass[a4paper,12pt,reqno]{amsart}

\usepackage{amscd,amssymb,amsmath,amsfonts,amsthm, mathrsfs,xspace}
\usepackage{graphicx}
\usepackage{verbatim,enumerate,paralist}
\usepackage{textcomp}
\usepackage[pagebackref,colorlinks]{hyperref}
\hypersetup{%
  urlcolor=blue,linkcolor=blue,citecolor=blue
}
\usepackage{pdfpages}

\usepackage[T1]{fontenc}
\usepackage[all,2cell,poly,knot,tips]{xy}
\UseAllTwocells
\CompileMatrices

\newcounter{Definitioncount}
\newtheorem{theorem}{Theorem}[section] 

\newtheorem{proposition}[theorem]{Proposition}

\theoremstyle{definition}
\newtheorem{remark}[theorem]{Remark}

\newtheoremstyle{fact}{\bigskipamount}{\medskipamount}{\upshape}{}{\itshape}{. }{ }{Fact}
\theoremstyle{fact}

\newtheoremstyle{genquest}{\bigskipamount}{\medskipamount}{\upshape}{}{\itshape}{. }{ }{General Question}
\theoremstyle{genquest}

\newtheoremstyle{step}{2\bigskipamount}{\medskipamount}{\upshape}{}{\itshape}{. }{ }{\underline{Step~\thestep}}
\theoremstyle{step}

\renewcommand{\thestep}{\arabic{step}}

\newcommand{\lra}{\longrightarrow}

\newcommand{\ldual}[1]{\mathord{{\let\nolimits\relax\sideset{^\wedge}{}{#1}}}}
\newcommand{\laction}[2]{\mathord{{\let\nolimits\relax\sideset{^{#1}}{}{#2}}}}
\newcommand{\conj}[2]{\mathord{{\let\nolimits\relax\sideset{^{#1}}{}{#2}}}}

\def\CA{{\mathscr A}}

\def\CE{{\mathscr E}}

\def\CM{{\mathscr M}}

\DeclareMathOperator{\Skew}{Skew}
\DeclareMathOperator{\Mnd}{Mnd}
\DeclareMathOperator{\LB}{LBrMon}
\DeclareMathOperator{\Span}{Span}

\newcommand{\M}{\ensuremath{\mathscr M}\xspace}

\newcommand{\LBM}{\ensuremath{\LB(\M)}\xspace}

\newcommand{\ox}{\otimes}
\newcommand{\x}{\times}

\newcommand{\op}{\ensuremath{^{\textnormal{op}}}}

\begin{document}

\title{A skew-duoidal Eckmann-Hilton argument and quantum categories}

\author{Stephen Lack}
\address{Department of Mathematics, Macquarie University, NSW 2109, Australia}
\email{steve.lack@mq.edu.au}

\author{Ross Street}
\address{Department of Mathematics, Macquarie University, NSW 2109, Australia}
\email{ross.street@mq.edu.au}

\thanks{Both authors gratefully acknowledge the support of the Australian Research Council Discovery Grant DP1094883; Lack acknowledges with equal gratitude the support of an Australian Research Council Future Fellowship}



\dedicatory{Dedicated to George Janelidze on his sixtieth birthday}



\begin{abstract}
\noindent A general result relating skew monoidal structures and monads is proved. 
This is applied to quantum categories and bialgebroids. 
Ordinary categories are monads in the bicategory whose morphisms are spans between sets. 
Quantum categories were originally defined as monoidal comonads on endomorphism objects in a particular monoidal bicategory $\CM$. 
Then they were shown also to be skew monoidal structures (with an appropriate unit) on objects in $\CM$. 
Now we see in what kind of $\CM$ quantum categories are merely monads. 

\end{abstract}

\maketitle

\section{Introduction}\label{intro}

The proof that higher homotopy groups are commutative was abstracted to the statement that monoids in the category of monoids are commutative monoids. 
This is known as the Eckmann-Hilton argument \cite{EH}.

In a seminar talk \cite{Walters1984}, 
Bob Walters suggested looking at a 2-dimension\-al version of this argument where monoids are replaced by monoidal categories. 
Joyal-Street \cite{BTC} showed that monoidales (= pseudomonoids) in the 2-category of monoidal categories and strong monoidal functors were braided monoidal categories.
They also pointed out that, repeating the process, monoidales in the 2-category of braided monoidal categories and braided strong monoidal functors were symmetric monoidal categories. Also, stabilization occurs at that stage: it is symmetric monoidal categories from there onwards. These facts together constitute the Eckmann-Hilton argument for monoidal categories; here, we shall be particularly interested in the fact that a monoidale in the 2-category of braided monoidal categories is a symmetric monoidal category. 

If in the above strong monoidal functors are replaced bymere (lax)  monoidal functors, no such collapsing or stabilization occurs. 
Monoidales in the 2-category of monoidal categories and monoidal functors are called ``2-monoidal categories'' in \cite{AguiarMahajan} and ``duoidal categories'' in \cite{dlVP} and \cite{Tanduoidal}.    

Recently Kornel Szlach\'anyi \cite{Szl2012} has excited our investigations \cite{SkMon} and \cite{scc} into {\em skew} monoidal categories. These are defined similarly to monoidal categories, except
that the morphisms expressing the associativity and unit laws
are not required to be invertible. The paper \cite{Szl2012} explained
the relationship between skew monoidal categories and bialgebroids; this was extended in \cite{SkMon} to the case of quantum
categories in place of bialgebroids.
The question therefore arises as to whether there might be an  Eckmann-Hilton-like argument in the skew context. Given the title of the paper, it will come as no surprise that this is the case; equally, given the non-invertibility inherent in the notion of skew monoidal category it should come as no suprise that what we have found is rather less tight than is the case for monoidal categories. Our result is Theorem~\ref{thm:EH} below; see also Remark~\ref{rmk:EH} for a discussion of the sense in which it should be seen as an Eckmann-Hilton result. Then in Section~\ref{smbc} we generalize this to the case of internal structures in a symmetric monoidal bicategory \M.

Not that we were led to the above considerations directly! 
We began with our main application to quantum categories.
Since \cite{Ben1967}, we have known that ordinary categories are monads in the bicategory $\Span$ whose morphisms are spans between sets. 
Quantum categories were originally defined in \cite{QCat} as monoidal comonads on endomorphism objects in a particular monoidal bicategory $\CM$ of comonoids and comodules. 
When $\CM$ is $\Span$, these are equivalent to ordinary categories. 
As mentioned in the previous paragraph, quantum categories in $\CM$ 
were shown in \cite{SkMon} also to be equivalent to skew monoidal structures 
(with an appropriate unit) on objects in $\CM$. 

The starting point of the present paper was a question by George Janelidze at the Category Theory Conference CT2009 in Calais, France. 
At the end of the second author's lecture, George asked why the definition of quantum category was so complicated. 
In his own lecture, George suggested studying monads in the bicategory of comonoids and comodules. 
This naturally leads to the question: in what kind of $\CM$ are quantum categories merely monads? 
We shall answer this in Section~\ref{qccc}.

\enlargethispage{\baselineskip}

\section{The categorical level}\label{catlevel}

As mentioned in the introduction, a duoidal (or 2-monoidal) 
category is a monoidale in the monoidal 2-category of monoidal categories,
strong monoidal functors, and monoidal natural transformations.
See any of \cite{AguiarMahajan, BatMarkl, Tanduoidal, dlVP} for 
a more explicit definition. 

Our notation for skew monoidales is to write $(A,i,p)$, where $A$ is the underlying object, $p$ is the multiplication $A\ox A\to A$, and $i$ is the unit $I\to A$. In the case of skew monoidales in {\bf Cat} --- that is, of skew monoidal categories --- the domain of $p$ is the product $A\x A\to A$, and $p$ gives the tensor product of $A$; while the domain of $i$ is the terminal category $1$, and we may identify $i$ with its image, the unit object of $A$. The structure morphisms are invariably called $\alpha$, $\lambda$, and $\rho$, and are omitted from the notation $(A,i,p)$.

A {\em skew duoidal category} $(A,k,m,i,p)$  is a skew monoidale in the 
2-category of skew monoidal categories, opmonoidal functors, and opmonoidal natural transformations.
So we have two skew monoidal categories  $(A,i,p)$ and $(A,k,m)$ such that 
$k\colon 1\lra A$ and $m\colon A\x A \lra A$ and the constraints are opmonoidal 
with respect to $(A,i,p)$.  Apart from the two skew monoidal categories, the extra data involved are four natural tranformations
$$\xymatrix{
AAAA \ar[r]^{mm}_{~}="1" \ar[d]_{1c1} & AA \ar[dd]^p  &
11 \ar[r]^{!}_{~}="3" \ar[d]_{ii} & 1 \ar[d]^{i} & 11 \ar[r]^{kk}_{~}="5" \ar[d]_{!} & AA \ar[d]^{p} & 1 \ar@{=}[d] \ar@{=}[r]_{~}="7" & 1 \ar[d]^{i} \\
AAAA \ar[d]_{pp} && AA \ar[r]_{m}^{~}="4" & A & 1 \ar[r]_{k}^{~}="6" & A & 1 \ar[r]_{k}^{~}="8" & A \\
AA \ar[r]_{m}^{~}="2" & A 
\ar@{=>}"2";"1"_{m_2} 
\ar@{=>}"4";"3"_{m_0} 
\ar@{=>}"6";"5"_{k_2} 
\ar@{=>}"8";"7"_{k_0} 
}$$ 
where we have omitted the tensor product symbol $\ox$ to save space. These natural transformations are 
subject to a long list of conditions which we shall not write out in full, but describe as follows:
\begin{enumerate}
\item there is an associativity condition for $m_2$ which involves the map $\alpha$ associated to $(A,i,p)$;
\item two conditions stating that $m_0$ is a unit for $m_2$, and involving the $\lambda$ and $\rho$ for $(A,i,p)$;
\item an associativity condition for $k_2$, once again involving the $\alpha$ for $(A,i,p)$;
\item two unit conditions for $k_0$ involving the $\lambda$ and $\rho$ for $(A,i,p)$;
\item two conditions stating that the $\alpha$ for $(A,k,m)$ is opmonoidal, one of which involves $m_2$ and the other $m_0$;
\item two conditions stating that the $\lambda$ for $(A,k,m)$ is opmonoidal, one of which involves $m_2$ and $k_2$, the other $m_0$ and $k_0$;
\item and two similar conditions stating that the $\rho$ for $(A,k,m)$ is opmonoidal.
\end{enumerate}

An {\em opmonoidal monad} is a monad in the 2-category of monoidal categories, opmonoidal functors, and opmonoidal natural transformations. We typically write $\eta$ for the unit and $\mu$ for the multiplication of a monad $T$, and we write $T_2$ and $T_0$ for the opmonoidal structure: here $T_0$ consists of a single map $TI\to I$, while $T_2$ consists of a natural family of morphisms $T(A\ox B)\to TA\ox TB$.

We saw in \cite{SkMon} that such an opmonoidal monad $(T,\eta,\mu,T_0,T_2)$ determines a skew monoidal category $(\CA,I,\ast)$, with the same unit $I$, via the formulas
\begin{eqnarray*}
A\ast B = TA\otimes B \ ,
\end{eqnarray*} 
\begin{eqnarray*}
\xymatrix @C3pc {
(A*B)*C \ar@{=}[d] \ar[rr]^{\alpha_{A,B,C}} && A*(B*C) \ar@{=}[d] \\
T(TA\ox B)\ox C \ar[r]^-{v_{A,B}\ox1} & 
(TA\ox TB)\ox C \ar[r]^-{\alpha_{TA,TB,C}} & TA\ox(TB\ox C) 
}
\end{eqnarray*} 
where $v_{A,B}$ is the ``fusion operator''
\begin{eqnarray*}
\xymatrix @C3pc {
T(TA\ox B) \ar[r]^{T_2} & TTA\ox TB \ar[r]^{\mu_A\ox1} & 
TA\ox TB }
\end{eqnarray*}  
and the unit constraints $\lambda_A\colon I*A\to A$ and 
$\rho_A\colon A\to A*I$ are given by  the composites
\begin{eqnarray*}
\xymatrix @R1pc {
I*A \ar@{=}[r] & TI\ox A \ar[r]^{T_0\ox1} & I\ox A \ar[r]^{\lambda_A}
& A \\
A \ar[r]^{\eta_A} & TA \ar[r]^{\rho_{TA}} & TA\ox I \ar@{=}[r] & A*I . }  
\end{eqnarray*}
The extra point to be made here is that, if $(\CA,I,\otimes)$ is lax braided, we obtain a skew duoidal category via the product and unit
maps
$$\xymatrix @R1pc {
(A,I,\ox)\x(A,I,\ox) \ar[r]^-{(*,\gamma)} & (A,I,\ox) \\
1 \ar[r]^{(I,\mu)} & (A,I,\ox) 
}$$
in which the middle-of-four morphism $\gamma$ is given by 
$$\xymatrix @C2pc {
(A\ox C)*(D\ox B) \ar[rr]^{\gamma_{A,C,B,D}} \ar@{=}[d] 
&& (A*D)\ox(C*B) \ar@{=}[d] \\
T(A\ox C)\ox(D\ox B) \ar[r]^-{T_2\ox1} &
(TA\ox TC)\ox (D\ox B) \ar[r]^{\gamma} & 
(TA\ox D)\ox(TC\ox B) }$$
and $\mu\colon I*I\to I$ is given by $T_0$. Here the $\gamma$ appearing at the bottom of the diagram is the middle-of-four morphism arising from the lax braiding on $(\CA,I,\ox)$. 

\begin{theorem}\label{thm:EH}
Let $(\CA,I,\otimes)$ be a lax-braided monoidal category. The assignment just described is an equivalence between opmonoidal monads $(T,\eta,\mu,T_0,T_2)$ on 
$(\CA,I,\otimes)$ and those skew duoidal structures 
$(\CA,I,\ast , I, \otimes)$ with $(\CA,I,\otimes)$ as the second of the two monoidal structures, for which the following composite is invertible.
\begin{eqnarray}\label{skduoidalcond}
\xymatrix{
A*B \ar[r]^-{\rho_A*\lambda^{-1}_B} & 
(A\ox I)*(I\ox B) \ar[r]^-{\gamma} & 
(A*I)\ox(I*B) \ar[r]^-{1\ox\lambda} &
(A*I)\ox B } 
\end{eqnarray}  
\end{theorem}
\begin{proof}
Given a skew duoidal category of the form $(\CA,I,\ast , I, \otimes)$ with \eqref{skduoidalcond} invertible, define an endofunctor $T\colon\CA \lra \CA$ by $TA=A\ast I$.
Put $\eta_A$ equal to $\rho_A\colon  A \lra A\ast I=TA$, and put $\mu_A \colon TTA \lra TA$ equal to the composite 
$$\xymatrix{
(A*I)*I \ar[r]^{\alpha} & A*(I*I) \ar[r]^-{1*\lambda_I} & A*I .}$$
This defines a monad $(T,\eta ,\mu )$ on $\CA$. The opmonoidal
structure is given by 
$$\xymatrix @R1pc {
T(A\ox B) \ar@{=}[d] \ar[rr]^{T_2} && TA\ox TB \ar@{=}[d] \\
 (A\ox B)*I \ar[r]_-{1*\rho_I} & 
(A\ox B)*(I\ox I) \ar[r]_-{\gamma} & (A*I)\ox (B*I)  \\
TI \ar@{=}[dr] \ar[rr]^{T_0} & & I \\
&  I*I \ar[ur]_{\lambda} . }$$
\end{proof}

\begin{remark}\label{rmk:EH}
As was mentioned in the introduction, one version of the Eckmann-Hilton argument states that to give to a braided monoidal category a further monoidal structure (in the 2-category of braided monoidal categories and braided strong monoid\-al functors) is actually not further structure, but just the requirement that the braided monoidal category be symmetric. We regard Theorem~\ref{thm:EH} as a generalization of this fact. Start with a lax-braided monoidal category in place of a braided one, and then consider a further skew monoidal structure on it. This time this {\em does} give further structure, but provided that we require the composite \eqref{skduoidalcond} to be invertible, this further structure reduces to an opmonoidal monad on the lax-braided monoidal category. In the non-skew case, this opmonoidal monad would be the identity.
\end{remark}

\section{The symmetric monoidal bicategory context}\label{smbc}

In this section we internalize the results of the previous section,
working in a braided monoidal bicategory \M in the sense of
\cite{MbHa} . We write as
if \M were in fact a 2-category. The braiding is denoted by 
$c_{A,B} \colon A\otimes B \lra B\otimes A$.

We write $\Mnd(\M)$ for the 2-category
of monads in \M , and $\Mnd^*(\M)$ for the bicategory
$\Mnd(\M\op)\op$; the objects of $\Mnd^*(\M)$ are still
just the monads in \M, but the 1-cells are the opmorphisms
of monads: these are similar to morphisms of monads except
that the direction of the 2-cell involved in the definition is 
reversed \cite{FTM}. (The definition of $\Mnd^*(\M)$  does not 
use the monoidal structure of \M.)

We also write $\Skew(\M)$ for the 2-category of skew monoidales,
opmonoidal morphisms, and monoidal natural transformations.
(This uses the monoidal structure of \M, but not the braiding.)

If \M is in fact braided, then $\Skew(\M)$ is also monoidal,
and so we can define monoidales and skew monoidales there.
A skew monoidale in $\Skew(\M)$ consists of skew
monoidales $(A,i,p)$ and $(A,k,m)$ such that $k$, $m$,
and the structure 2-cells $\alpha$, $\lambda$, and $\rho$ 
for $(A,k,m)$ are opmonoidal with
respect to $(A,i,p)$; such a structure $(A,k,m,i,p)$ 
 is what we call a skew duoidale in \M.

We also use the full braided monoidal structure of \M when we define $\LB(\M)$ to be the monoidal 2-category
of lax braided monoidales in \M, with opmonoidal morphisms. For an object $A\in\LB(\M)$, we write $\nabla\colon A\ox A\to A$
for the multiplication, $j\colon I\to A$ for the unit, and 
$\gamma$ for the 2-cell, defined using the lax braiding, which
 expresses the fact that $\nabla$ is itself opmonoidal. (The 
remaining structure is generally not mentioned explicitly.)

A lax braided monoidale $(A,i,p)$ determines a skew 
duoidale $(A,i,p,i,p)$.

A morphism in $\LB(\M)$ from $A$ to $B$ involves
a 1-cell $f\colon A\to B$ and 2-cells 
$$\xymatrix{
A\ox A \ar[r]^{f\ox f}_{~}="2" \ar[d]_{\nabla} & B\ox B \ar[d]^{\nabla} & 
& I \ar[dr]^{j}_{~}="4" \ar[dl]_{j}^{~}="3" \\
A \ar[r]_f^{~}="1" & B , & A \ar[rr]_f && B .
\ar@{=>}"1";"2"^{f_2} 
\ar@{=>}"3";"4"^{f_0} }$$

There is a 2-functor $R\colon \Skew(\M)\to\Mnd(\M)$ sending 
a skew monoidale $(A,i,m)$ to the monad 
$$\xymatrix{
A \ar[r]^-{1\ox i} & A\ox A \ar[r]^-{m} & A }$$
with multiplication
$$\xymatrix{
A \ar[d]_{1i} \ar[r]^{1i} & AA \ar[d]_{1i1}^{~}="3" \ar@/^1pc/[dr]^1_{~}="4"  & \ar@{}[d]_{~}="5"  \\
AA \ar[d]_{m} \ar[r]_{11i} & AAA \ar[d]_{m1}^{~}="1" \ar[r]^{1m} & AA \ar[d]^{m}_{~}="2"  \\
A \ar[r]_{1i} & AA \ar[r]_{m} & A 
\ar@{=>}"1";"2"^{\alpha} \ar@{=>}"3";"5"^-{1\lambda} }$$
and with unit $\rho$. 

In the diagram above we have omitted the
tensor products to save space; we have also not explicitly named 
the invertible 2-cells coming from pseudofunctoriality of the tensor product on \M. We shall continue to follow this practice throughout the paper, also not naming certain isomorphisms which form part of the ``ambient structure'' in \M or $\LB(\M)$, such as the associativity isomorphisms $\nabla.\nabla1\cong \nabla.1\nabla$ for a lax braided monoidale.

Since $\LB(\M)$ is a monoidal bicategory, there is a corresponding
2-functor $$R\colon  \Skew(\LB(\M))\to\Mnd^*(\LB(\M)) \ .$$

On the other hand there is a 2-functor 
$$T\colon \Mnd^*(\LBM)\to\Skew(\LBM)$$ 
sending a monad $(A,t)$ to the skew monoidale with multiplication 
$$\xymatrix{A\ox A \ar[r]^{t\ox 1} & A\ox A \ar[r]^-{\nabla} & A }$$
with unit 
$j\colon 1\to A$, with associativity constraint $\alpha$ given by 
$$\xymatrix{
AAA \ar[r]^{1t1}_{~}="1" \ar[d]_{t11} & AAA \ar[r]^{1\nabla} \ar[d]^{t11} & AA \ar[d]^{t1} \\
AAA \ar[r]^{t11}="2"_{~}="3" \ar[d]_{\nabla1} & 
AAA \ar[r]^{1\nabla}_{~}="5" \ar[d]^{\nabla1} & AA \ar[d]^{\nabla} \\
AA \ar[r]_{t1}^{~}="4" & AA \ar[r]_{\nabla} & A 
\ar@{=>}"2";"1"_{\mu t1} \ar@{=>}"4";"3"_{t_21} 
}$$ 
with right unit constraint $\rho$ given by 
$$\xymatrix{
A \ar[r]^{1j} \dtwocell_1^t{^\eta} & AA \ar[d]^{t1} \\
A \ar[r]^{1j} \ar[dr]_1 & AA \ar[d]^{\nabla} \\
& A }$$
and with left unit constraint $\lambda$ given by 
$$\xymatrix @C3pc {
AA \ar[d]_{t1}^{~}="1" & A \ar[l]_{j1} \ar@/^1pc/[dl]^{j1} \ar@{}[d]_{~~~}="2" \ar@/^2pc/[ddl]^1 \\
AA \ar[d]_{\nabla} & \\
A .
\ar@{=>}"1";"1"+/va(0) 2.5pc/^{t_01} 
}$$

Now consider the composite $RT$. This sends a monad $t$ on $A$
to a monad on $A$ whose underlying 1-cell is the right hand 
composite in the diagram
$$\xymatrix{
A \ar[r]^{1j} \ar[d]_{t} & AA \ar[d]^{t1} \\
A \ar[r]^{1j} \ar[dr]_{1} & AA \ar[d]^{\nabla} \\
& A }$$
(in which the two regions commute up to isomorphisms coming
from pseudofunctoriality of the tensor in \LBM, and the right
unit constraint for the lax braided monoidal structure on $A$). 
Compatibility of this isomorphism with the units for the monads
holds by definition of the monad on the right, and a straightforward calculation gives compatibility with the multiplications for the
monads as well. 

Thus we have an isomorphism $RT\cong 1$, whose component at an object $(A,t)$ of $\Mnd(\LBM)$ is the morphism 
$(A,t)\to RT(A,t)$ of monads which is the identity $A\to A$ 
equipped with the isomorphism of monads described above. 

Now consider the other composite $TR$. Suppose that 
$A = (A,i,m)$ is a skew monoidale in \LBM, for which $i\colon 1\to A$
is strong (op)monoidal, as will always be the case for 
an object in the image of $T$. In particular, we have $i\cong j$,
so we may as well take $i$ to be $j$ itself. 

For such an $A$, we have a 2-cell
$$\xymatrix{
AAA \ar[r]^{m1}_{~}="2" \ar[d]_{1\nabla} & AA \ar[d]^{\nabla} \\
AA \ar[r]_{m}^{~}="1" & A 
\ar@{=>}"1";"2"^{\psi} }$$
given by the composite 
$$\xymatrix{
AAA \ar[dr]_{11j1} \ar@/^1pc/[drr]^{m1}_{~}="3" \ar@/_1pc/[ddr]_{1\nabla} \\
& AAAA \ar[r]^{mm}_{~}="2" \ar[d]_{\nabla^2} \ar@{=>}"3"_(0.6){m\lambda} & 
AA \ar[d]^{\nabla} \\
& AA \ar[r]_{m}^{~}="1" & A 
\ar@{=>}"1";"2"^{m_2} 
}$$
where 
$\nabla^2=(\xymatrix{A^4 \ar[r]^{1c_{A,A}1} & A^4 \ar[r]^{\nabla\nabla} & A^2})$
is the multiplication on $A^2$.

\begin{proposition}\label{prop:psi-a}
  The 2-cell $\psi$ satisfies 
$$\xymatrix{
& A^3 \ar[r]^{1\nabla}  \ar[dr]^{\nabla1} & A^2 \ar[dr]^{\nabla} &&&
& A^3 \ar[r]^{1\nabla} & A^2 \ar[dr]^{\nabla} \\
A^4 \ar[ur]^{m11} \ar[dr]_{1\nabla1} & \ar@{}[d]|{~}="1" \ar@{}[u]|{~}="2" & A^2 \ar[r]^{\nabla} & A & = & 
A^4 \ar[ur]^{m11} \ar[r]^{11\nabla} \ar[dr]_{1\nabla1} &
 A^3 \ar[ur]^{m1} \ar[rd]_{1\nabla} & \ar@{}[d]|{~}="3" \ar@{}[u]|{~}="4" & A \\
& A^3 \ar[r]_{1\nabla} \ar[ur]^{m1}  
& A^2 \ar[ur]_{m} \ar@{=>}[u]^{\psi}  &&&
& A^3 \ar[r]_{1\nabla} & A^2 \ar[ur]_{m} 
\ar@{=>}"1";"2"^{\psi1} 
\ar@{=>}"3";"4"^{\psi} 
}$$
\end{proposition}

\proof
Use naturality, coassociativity of $m_2$, monoidale axioms
for $(A,\nabla,j)$, and opmonoidality of $\lambda$.
\endproof

\begin{proposition}\label{prop:psi-alpha}
  The 2-cell $\psi$ satisfies 
$$\xymatrix{
& A^3 \ar[r]^{m1}  \ar[dr]_{1\nabla} & A^2 \ar[dr]^{\nabla} &&&
& A^3 \ar[r]^{m1} & A^2 \ar[dr]^{\nabla} \\
A^4 \ar[ur]^{1m1} \ar[dr]_{11\nabla} & \ar@{}[d]|{~}="1" \ar@{}[u]|{~}="2" & A^2 \ar[r]^{m} \ar@{=>}[u]^{\psi} & A & = & 
A^4 \ar[ur]^{1m1} \ar[r]^{m11} \ar[dr]_{11\nabla} &
 A^3 \ar[ur]_{m1} \ar[rd]_{1\nabla} \ar@{=>}[u]^{\alpha1} & \ar@{}[d]|{~}="3" \ar@{}[u]|{~}="4" & A \\
& A^3 \ar[r]_{m1} \ar[ur]^{1m} 
& A^2 \ar[ur]_{m} \ar@{=>}[u]^{\alpha}  &&&
& A^3 \ar[r]_{m1} & A^2 \ar[ur]_{m} 
\ar@{=>}"1";"2"^{1\psi} 
\ar@{=>}"3";"4"^{\psi} }$$
\end{proposition}

\proof
Rewrite $m_2$ in terms of $(1m)_2$, then use 
naturality and opmonoidality of $\alpha$, and the skew 
monoidale axioms for $(A,m,j)$.
\endproof

Restricting $\psi$ along $1j1\colon AA\to AAA$ and using the isomorphism
$1\nabla.1j1\cong 1$ gives a 2-cell
$$\xymatrix{
AA \ar[r]^{t1}_{~}="2" \ar@{=}[d] & AA \ar[d]^{\nabla} \\
AA \ar[r]_{m}^{~}="1" & A 
\ar@{=>}"1";"2"^{\chi} }$$
where $t$ is the induced monad, given by $m.1j$. 

\begin{proposition}\label{prop:psi-chi}
  The 2-cells $\psi$ and $\chi$ are linked via the equation
$$\vcenter{\vbox{\xymatrix{
A^2 \ar@/^2pc/[rr]^{t1} & A^3 \ar[r]_{1\nabla} \ar[d]^{\nabla 1} & 
A^2 \ar[d]^{\nabla} \\
A^3 \ar[u]^{1\nabla} \ar[ur]^{t11}_(.6){~}="2" \ar[r]_{m1}^(.7){~}="1"  \ar[dr]_{1\nabla} & 
A^2 \ar[r]^{\nabla} & A \\
& A^2 \ar[ur]_{m} \ar@{=>}[u]_{\psi} 
\ar@{=>}"1";"1"+/va(90) 1.5pc/^{\chi 1} 
}}}=
\vcenter{\vbox{\xymatrix{
A^2 \ar@/^2pc/[rr]^{t1} \ar@{=}[ddr] & & 
A^2 \ar[d]^{\nabla} \\
A^3 \ar[u]^{1\nabla} \ar[dr]_{1\nabla} & \ar@{}[r]^(.2){~}="1" \ar@{}[ur]_(.3){~}="2" & A \\
& A^2 \ar[ur]_{m}
\ar@{=>}"1";"1"+/va(90) 3pc/^{\chi} 
}}}$$
\end{proposition}

\proof
Take the equality in Proposition~\ref{prop:psi-a} and restrict along
the arrow $1j11\colon A^3\to A^4$.
\endproof

We shall show that $\chi$ is compatible with the associativity and 
unit constraints and so makes the identity morphism $1\colon A\to A$
into a morphism of skew monoidales from $(A,m)$ to 
$TR(A,m)$. 

Restricting $\lambda$ along $j\colon 1\to A$ gives 
$m_0\colon m.jj\to j$; it follows that $\chi$ is compatible with 
the right unit constraints. Compatibility with the left 
unit constraints once again uses the fact that $\lambda.j=m_0$,
along with the fact that $\lambda$ is opmonoidal.

It remains to check that $\chi$ is compatible with the 
associativity constraints. This says that the composites

\begin{equation}
  \label{eq:1}
\vcenter{\vbox{
 \xymatrix{
A^3 \ar[d]^{t11} \ar@/_3pc/[dd]_{m1}^(.7){~}="1" \ar[r]^{1t1} & 
A^3 \ar[r]^{1\nabla} & A^2 \ar[d]^{tA} \\
A^3 \ar[d]^{\nabla 1} 
&& A^2 \ar[d]^{\nabla} \\
A_2 \ar[r]^{t1} \ar@/_2pc/[rr]_{m}^(.4){~}="2" & 
A^2 \ar[r]^{\nabla}_{~}="3" \ar@{=>}[ur]^{\alpha'} 
\ar@{=>}"2";"2"+/va(45)+1.5pc/^(.4){\chi} 
\ar@{=>}"1";"1"+/va(45)+1.5pc/^{\chi1} 
& A 
}}}
=
\vcenter{\vbox{\xymatrix{
& A^3 \ar[dr]^{1\nabla}_{~}="4" \\
A^3 \ar[dd]_{m1} \ar[ur]^{1t1} \ar[rr]_{1m}="1"^{~}="3" & 
\ar@{=>}"3";"3"+/va(45) 1.8pc/^{1\chi} & A^2 \ar[dr]^{t1}_{~}="5" \ar[dd]_m="7"^{~}="6" \\
&& \ar@{=>}"6";"6"+/va(45) 1.8pc/_{\chi} & A^2 \ar[dl]^{\nabla} \\
A^2 \ar[rr]_{m}^{~}="2" & \ar@{=>}"2";"7"^{\alpha} & A }}}
\end{equation}
are equal, where $\alpha'$ is the associativity constraint for 
$TR(A,m,i)$, given by 
$$\xymatrix @C4pc {
A^3 \ar[r]^{A^2jA} \ar[d]_{AjA^2} & 
A^4 \ar[r]^{AmA}_{~}="1" \ar[d]_{AjjA^3} & 
A^3 \ar[r]^{A\nabla} \ar[d]^{AjA^2} & A^2 \ar[d]^{AjA} \\
A^4 \ar[r]^{A^2jAjA} \ar[d]_{mA^2} & 
A^6 \ar[r]^{AmmA}="2"_{~}="3" \ar[d]_{mA^3} &
A^4 \ar[r]^{A^2\nabla} \ar[d]^{mA^2} & A^3 \ar[d]^{mA} \\
A^3 \ar[r]^{AjAjA} \ar[d]_{\nabla A} & 
A^5 \ar[r]^{mmA}="4"_{~}="5" \ar[d]_{\nabla^2A} & 
A^3 \ar[r]^{A\nabla} \ar[d]^{\nabla A} & A^2 \ar[d]^{\nabla} \\
A^2 \ar[r]_{AjA} & A^3 \ar[r]_{mA}^{~}="6" & A^2 \ar[r]_\nabla & A 
\ar@{=>}"2";"1"^{Am_0mA} 
\ar@{=>}"4";"3"^{\alpha mA} 
\ar@{=>}"6";"5"^{m_2 A} 
}$$
where $\nabla^2\colon A^4\to A^2$ denotes the multiplication on 
$AA$, defined using $\nabla$ and the braiding. We can 
rewrite this as
$$\xymatrix @C4pc {
A^3 \ar[rr]^{A^2jA} \ar[d]_{AjA^2} & &
A^4 \ar[r]^{AmA}_{~}="1" \ar[d]_{AjjA^3} & 
A^3 \ar[r]^{A\nabla} \ar[d]^{AjA^2} & A^2 \ar[d]^{AjA} \\
A^4 \ar[r]^{A^3jA} \ar[d]_{mA^2}  & A^5 \ar[r]^{A^2jA^3} \ar[d]_{mA^3} & 
A^6 \ar[r]^{AmmA}="2"_{~}="3" \ar[d]_{mA^4} &
A^4 \ar[r]^{A^2\nabla} \ar[d]^{mA^2} & A^3 \ar[d]^{mA} \\
A^3 \ar[r]^{A^2jA} \ar[d]_{\nabla A} & A^4 \ar[r]^{AjA^3}  \ar[dr]_{\nabla A^2} &
A^5 \ar[r]^{mmA}="4"_{~}="5" \ar[d]_{\nabla^2A} & 
A^3 \ar[r]^{A\nabla} \ar[d]^{\nabla A} & A^2 \ar[d]^{\nabla} \\
A^2 \ar[rr]_{AjA} && A^3 \ar[r]_{mA}^{~}="6" & A^2 \ar[r]_\nabla & A 
\ar@{=>}"2";"1"^{Am_0mA} 
\ar@{=>}"4";"3"^{\alpha mA} 
\ar@{=>}"6";"5"^{m_2 A} 
}$$
and now the left hand side of \eqref{eq:1} becomes
$$\xymatrix @C4pc {
& A^3 \ar[rr]^{A^2jA} \ar[d]^{AjA^2} \ar[ddl]_1 & &
A^4 \ar[r]^{AmA}_{~}="1" \ar[d]_{AjjA^3} & 
A^3 \ar[r]^{A\nabla} \ar[d]^{AjA^2} & A^2 \ar[d]^{AjA} \\
& A^4  \ar[r]^{A^3jA} \ar[d]^{mA^2}_{~}="8"  \ar[dl]^{A\nabla A} &  
A^5 \ar[r]^{A^2jA^3} \ar[d]_{mA^3} & 
A^6 \ar[r]^{AmmA}="2"_{~}="3" \ar[d]_{mA^4} &
A^4 \ar[r]^{A^2\nabla} \ar[d]^{mA^2} & A^3 \ar[d]^{mA} \\
A^3 \ar[dr]_{mA}^{~}="7" \ar@{=>}"7";[r]_{\psi} & 
A^3 \ar[r]^{A^2jA} \ar[d]^{\nabla A} & A^4 \ar[r]^{AjA^3}  \ar[dr]_{\nabla A^2} &
A^5 \ar[r]^{mmA}="4"_{~}="5" \ar[d]_{\nabla^2A} & 
A^3 \ar[r]^{A\nabla} \ar[d]^{\nabla A} & A^2 \ar[d]^{\nabla} \\
&A^2 \ar[rr]_{AjA} \ar[drrr]_{1} && 
A^3 \ar[r]_{mA}^{~}="6" \ar[dr]_{A\nabla} & A^2 \ar[r]_\nabla  & A \\
&&&& A^2 \ar[ur]_{m}  \ar@{=>}[u]_{\psi} 
\ar@{=>}"2";"1"^{Am_0mA} 
\ar@{=>}"4";"3"^{\alpha mA} 
\ar@{=>}"6";"5"^{m_2 A} 
}$$
which can also be written as 
$$\xymatrix @C4pc {
A^3 \ar[rr]^{A^2jA} \ar[d]^{AjAjA} \ar@/_2pc/[ddd]|{A^2jA} \ar@/_4pc/[dddd]_1 & &
A^4 \ar[r]^{AmA}_{~}="1" \ar[d]_{AjjA^3} & 
A^3 \ar[r]^{A\nabla} \ar[d]^{AjA^2} & A^2 \ar[d]^{AjA} \\
A^5 \ar[rr]^{A^2jA^3} \ar[dr]^{mA^3} \ar[dd]^{A\nabla A^2} && 
A^6 \ar[r]^{AmmA}="2"_{~}="3" \ar[d]_{mA^4} &
A^4 \ar[r]^{A^2\nabla} \ar[d]^{mA^2} & A^3 \ar[d]^{mA} \\
 & 
A^4 \ar[r]^{AjA^3}  \ar[dr]_{\nabla A^2} &
A^5 \ar[r]^{mmA}="4"_{~}="5" \ar[d]_{\nabla^2A} & 
A^3 \ar[r]^{A\nabla} \ar[d]^{\nabla A} & A^2 \ar[d]^{\nabla} \\
A^4 \ar[rr]_{mA^2}^{~}="7" \ar@{=>}"7";[ur]_{\psi A} \ar[d]^{A^2\nabla} &&
A^3 \ar[r]_{mA}^{~}="6" \ar[dr]_{A\nabla} & A^2 \ar[r]_\nabla  & A \\
A^3 \ar[rrr]^{mA} &&& A^2 . \ar[ur]_{m}  \ar@{=>}[u]_{\psi} 
\ar@{=>}"2";"1"^{Am_0mA} 
\ar@{=>}"4";"3"^{\alpha mA} 
\ar@{=>}"6";"5"^{m_2 A} 
}$$
On the other hand, the right hand side of \eqref{eq:1} is 
$$\xymatrix{
& A^3 \ar[dr]^{1\nabla}_{~}="7" \\
A^3 \ar[dd]_{m1} \ar[ur]^{1t1} \ar[rr]_{1m}="1"^{~}="6" & 
\ar@{=>}"6";"6"+/va(45) 2pc/^{1\chi} & A^2 \ar[dr]^{t1}_{~}="2" \ar[dd]_m="4"^{~}="3" \\
&& \ar@{=>}"3";"3"+/va(45) 2pc/^{\chi} & A^2 \ar[dl]^{\nabla} \\
A^2 \ar[rr]_{m}^{~}="5" & \ar@{=>}"5";"5"+/va(45) 3pc/^{\alpha} & A }
$$
and now  using Proposition~\ref{prop:psi-chi} this is 
$$\xymatrix{
&&&& A^2 \ar[dddl]^{t1} \\ 
&& A^3 \ar[dl]_{1\nabla}^{~} \ar[d]_(.55){m1}="4"^(.4){~}="5" \ar[r]^{t11}_(.4){~}="6" \ar@/^1pc/[urr]^{1\nabla} & 
A^3 \ar[dl]^{\nabla1}_{~} \ar[dd]_{1\nabla} \\
A^3 \ar[r]_{1m}^(.5){~}="7" \ar[d]_{m1} \ar@/^2pc/[urr]^{1t1} & 
A^2 \ar[d]_m="2"^(.3){~}="3" 
&
A^2 \ar[dl]^{\nabla}_{~} \\
A^2 \ar[r]_{m}^(.3){~}="1" & A & & A^2 \ar[ll]^{\nabla} 
\ar@{=>}"1";"1"+/va(45) 2pc/^{\alpha} 
\ar@{=>}"3";"3"+/va(45) 2pc/_{\psi} 
\ar@{=>}"5";"5"+/va(45) 1.5pc/_{\chi 1} 
\ar@{=>}"7";"7"+/va(45) 3.5pc/^{1\chi} 
}$$
which  by Proposition~\ref{prop:psi-alpha}  is the same as 
$$
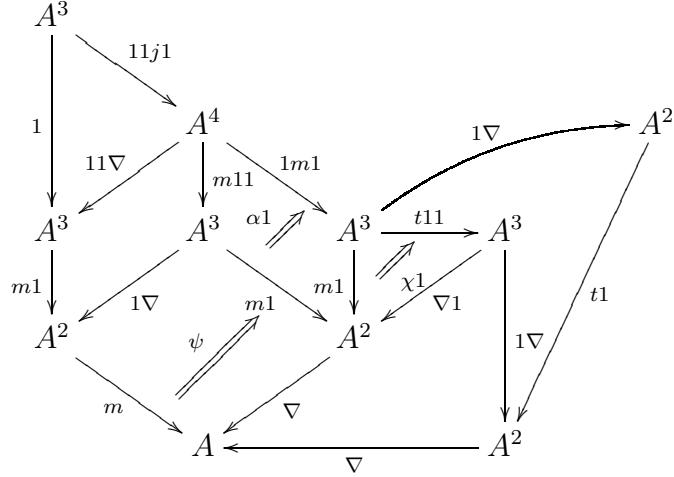
\xymatrix @C3pc {
A^3 \ar[dr]^{11j1} \ar[dd]_{1} \\
& A^4 \ar[dr]^{1m1}_(.8){~}="2" \ar[dl]_{11\nabla} \ar[d]^{m11} &&&
A^2 \ar[dddl]^{t1} \\
A^3 \ar[d]_{m1} & A^3 \ar[dr]_{m1}="4"^(.3){~}="1" \ar[dl]^{1\nabla}
& 
A^3  \ar[d]_{m1}^(.5){~}="5" \ar[r]^{t11}_(.4){~}="6" \ar@/^1pc/[urr]^{1\nabla} & 
A^3 \ar[dl]^{\nabla1}_{~} \ar[dd]^{1\nabla} \\
A^2  \ar[dr]_{m}^(.7){~}="3"  & &
A^2 \ar[dl]^{\nabla}_{~} \\
& A & & A^2 \ar[ll]^{\nabla} 
\ar@{=>}"1";"1"+/va(45) 2pc/^{\alpha1} 
\ar@{=>}"3";"3"+/va(45) 4pc/^{\psi} 
\ar@{=>}"5";"5"+/va(45) 2pc/_{\chi 1} 
}$$
which we can rewrite as 
$$
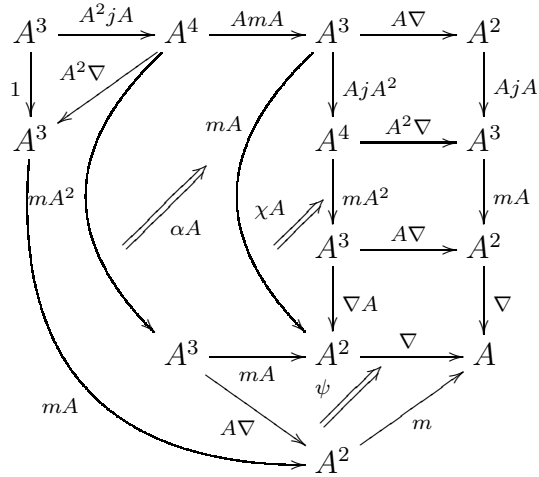
\xymatrix @C3pc {
A^3 \ar[r]^{A^2jA} \ar[d]_{1}   &
A^4 \ar[r]^{AmA}  \ar@/_3pc/[ddd]_{mA^2}^(.7){~}="1" \ar[dl]_{A^2\nabla} & 
A^3 \ar[r]^{A\nabla} \ar[d]^{AjA^2} \ar@/_3pc/[ddd]_(.3){mA}="2"^(.7){~}="3" & A^2 \ar[d]^{AjA} \\
A^3  \ar@/_4pc/[dddrr]_{mA} && 
A^4 \ar[r]^{A^2\nabla} \ar[d]^{mA^2}_{~}="4" & A^3 \ar[d]^{mA} \\
&&
A^3 \ar[r]^{A\nabla} \ar[d]^{\nabla A} & A^2 \ar[d]^{\nabla} \\
& A^3 \ar[r]_{mA}^{~}="6" \ar[dr]_{A\nabla}^(0.8){~}="7" & A^2 \ar[r]^\nabla  & A \\
&& A^2 \ar[ur]_{m}  
\ar@{=>}"1";"1"+/va(45) 4pc/_{\alpha A} 
\ar@{=>}"3";"3"+/va(45) 2.5pc/^{\chi A} 
\ar@{=>}"7";"7"+/va(45) 3pc/^(.4){\psi} 
}$$
and now \eqref{eq:1} will follow if we can prove

\begin{equation}
  \label{eq:2}
\xymatrix @C4pc {
A^2 \ar[rr]^{A^2j} \ar[d]^{AjAj} \ar@/_2pc/[ddd]_{A^2j} &&
A^3 \ar[r]^{Am}_{~}="1" \ar[d]_{AjjA^2} & 
A^2  \ar[d]^{AjA} \\
A^4 \ar[rr]^{A^2jA^2} \ar[dr]^{mA^2}_{~}="8" \ar[dd]^{A\nabla A} && 
A^5 \ar[r]^{Amm}="2"_{~}="3" \ar[d]_{mA^3} &
A^3 \ar[d]^{mA}  \\
 & 
A^3 \ar[r]^{AjA^2}  \ar[dr]_{\nabla A} &
A^4 \ar[r]^{mm}="4"_{~}="5" \ar[d]_{\nabla^2} & 
A^2 \ar[d]^{\nabla} \\
A^3 \ar[rr]_{mA}^(0.3){~}="7"   &&
A^2 \ar[r]_{m}^{~}="6" & A 
\ar@{=>}"2";"1"^{Am_0m} 
\ar@{=>}"4";"3"^{\alpha m} 
\ar@{=>}"6";"5"^{m_2} 
\ar@{=>}"7";"7"+/va(90) 3pc/_{\psi A}
}=
\xymatrix{ 
A^2 \ar[r]^{A^2j} & A^3 \ar[r]^{Am} \ar@/_3pc/[ddd]_{mA}^{~}="3" & 
A^2 \ar@/_3pc/[ddd]_{m}="4"^{~}="1" \ar[d]^{AjA} \\
&& A^3 \ar[d]^{mA}_{~}="2" \\
&& A^2 \ar[d]^{\nabla} \\
& A^2 \ar[r]_{m} & A 
\ar@{=>}"1";"2"^{\chi} 
\ar@{=>}"3";"4"^{\alpha} 
}
\end{equation}

The right hand side can be rewritten as 
$$\xymatrix{
A^2 \ar[r]^{A^2j}  & A^3 \ar[ddd]_1 \ar[r]^{Am} 
& A^2 \ar[d]^{AjA} \ar@/_2pc/[dd]_1 \\
&& A^3 \ar[d]^{A\nabla} \ar[dr]^{mA}_(.7){~}="2" \\
&& A^2 \ar[dr]_{m} \ar@{}[dr]_(.3){~}="4"^(.3){~}="1" & A^2 \ar[d]^{\nabla} \\ 
 & A^3 \ar[ur]^{Am} \ar[dr]_{mA}^(.6){~}="3" && A\\
&& A^2 \ar[ur]_{m} 
\ar@{=>}"1";"1"+/va(70) 1.8pc/_{\psi} 
\ar@{=>}"3";"3"+/va(70) 3pc/_{\alpha} 
}$$
and so, using one of the counit laws for the opmonoidal structure on 
$m$,  as the composite on the left in the following display,
which in turn can be written as the composite on the right.
$$\xymatrix @C3pc {
& & A^2 \ar[d]^{AjA}  \\
A^2 \ar[r]^{A^2j} &
A^3 \ar@/_3pc/[dd]_1 \ar[ur]^{Am}_(.6){~}="6" \ar[d]_{AjjA^2} & A^3 \ar[d]^{A\nabla} \ar[dr]^{mA}_(.6){~}="2" \\
& A^5 \ar[d]_{A\nabla^2} \ar[ur]|(.3){Amm}^{~}="5"_(.6){~}="8" & A^2 \ar[dr]_{m} \ar@{}[dr]_(.3){~}="4"^(.3){~}="1" & A^2 \ar[d]^{\nabla} \\ 
 & A^3 \ar[ur]_{Am}^{~}="7" \ar[dr]_{mA}^(.7){~}="3" && A\\
&& A^2 \ar[ur]_{m} 
\ar@{=>}"1";"1"+/va(70) 2pc/_{\psi} 
\ar@{=>}"3";"3"+/va(70) 4pc/_{\alpha} 
\ar@{=>}"7";"7"+/va(70) 2.5pc/^(.3){Am_2} 
\ar@{=>}"5";"5"+/va(70) 3pc/|(.5){Am_0m}
}
\xymatrix @C3pc {
& & A^2 \ar[d]^{AjA}  \\
A^2 \ar[r]^{A^2j} &
A^3 \ar@/_3pc/[ddd]_1 \ar[ur]^{Am}_(.6){~}="6" \ar[d]_{AjjA^2} & 
A^3 \ar[d]^{A^2jA} \ar@/^2pc/[ddr]^{mA}_(.45){~}="14" \\
& A^5 \ar[ur]_(.3){Amm}^(.6){~}="5" \ar[d]_{A^3jA^2} & 
A^4 \ar[d]_{\nabla^2} \ar[dr]_(.4){mm} \ar@{}[dr]^{~}="13"_(.65){~}="2"  \\
& A^6 \ar[d]_{\nabla^3} \ar[ur]^{AmAm}_(.6){~}="8" & 
A^2 \ar[dr]_{m} \ar@{}[dr]_(.3){~}="4"^(.4){~}="1" & A^2 \ar[d]^{\nabla} \\ 
 & A^3 \ar[ur]_{Am}^(.6){~}="7" \ar[dr]_{mA}^(.8){~}="3" && A \\
&& A^2 \ar[ur]_{m} 
\ar@{=>}"1";"1"+/va(70) 2pc/_{m_2} 
\ar@{=>}"3";"3"+/va(70) 4pc/_{\alpha} 
\ar@{=>}"7";"7"+/va(70) 2.5pc/^(.3){(Am)_2} 
\ar@{=>}"5";"5"+/va(70) 3pc/|(.5){Am_0m}
\ar@{=>}"13";"13"+/va(70) 2pc/_{m\lambda} 
}$$
By opmonoidality of $\alpha$, this is equal to the composite 
on the  left in the following display 
which by one of the unit axioms for the monoidale $(A,m,j)$
is equal to the diagram on the right.
$$\xymatrix @C3pc {
& & A^2 \ar[d]^{AjA}  \\
A^2 \ar[r]^{A^2j} &
A^3 \ar@/_3pc/[ddd]_1 \ar[ur]^{Am}_(.6){~}="6" \ar[d]_{AjjA^2} & 
A^3 \ar[d]^{A^2jA} \ar@/^2pc/[ddr]^{mA}_(.45){~}="14" \\
& A^5 \ar[ur]_(.3){Amm}^(.6){~}="5" \ar[d]_{A^3jA^2} & 
A^4 \ar[dr]^(.4){mm}="13"_(.25){~}="4"  \\
& A^6 \ar[d]_{\nabla^3} \ar[ur]^{AmAm} \ar[dr]^(.3){mAmA}_(.45){~}="8"^(.8){~}="3" & 
 & A^2 \ar[d]^{\nabla} \\ 
 & A^3  \ar[dr]_{mA}^(.2){~}="7" & 
A^4 \ar[ur]^{mm}_(.6){~}="2" 
\ar[d]^{\nabla^2} & A\\
&& A^2 \ar[ur]_{m}^(.6){~}="1" 
\ar@{=>}"1";"1"+/va(70) 3pc/^{m_2} 
\ar@{=>}"3";"3"+/va(70) 4.5pc/_{\alpha\alpha} 
\ar@{=>}"7";"7"+/va(70) 2pc/_(.3){(mA)_2} 
\ar@{=>}"5";"5"+/va(70) 3pc/|{Am_0m}
\ar@{=>}"13";"13"+/va(70) 2pc/_{m\lambda} 
}
\xymatrix @C3pc {
& & A^2 \ar[d]^{AjA}  \\
A^2 \ar[r]^{A^2j} \ar[d]|(.4){AjjA} \ar@/_2pc/[dd]_1 &
A^3  \ar[ur]^{Am}_{~}="6" \ar[d]|(.4){AjjA^2} & 
A^3  \ar@/^2pc/[ddr]^{mA}_(.45){~}="4" \\
A^4 \ar[dr]|{A^2jA^2j} \ar[d]_{\nabla^2} & 
A^5 \ar[ur]_{Amm}^{~}="5" \ar[d]|(.4){A^3jA^2} \ar@/^2pc/[ddr]^(0.3){mA^3}_(0.45){~}="14" & \ar@{}[r]_{~}="16" & \\
A^2 \ar[dr]_{A^2j} & A^6 \ar[d]_{\nabla^3}  \ar[dr]|{mAmA}_(.45){~}="8"^(.45){~}="13" & 
\ar@{}[r]^(.3){~}="15" 
 & A^2 \ar[d]^{\nabla} \\ 
 & A^3  \ar[dr]_{mA}^(.2){~}="7" & 
A^4 \ar[ur]^{mm}_(.6){~}="2"  \ar[d]^{\nabla^2} & A\\
&& A^2 \ar[ur]_{m}^(.6){~}="1" 
\ar@{=>}"1";"1"+/va(70) 3pc/^{m_2} 
\ar@{=>}"7";"7"+/va(70) 1.5pc/_(.3){(mA)_2} 
\ar@{=>}"5";"5"+/va(70) 3pc/|{Am_0m}
\ar@{=>}"13";"13"+/va(70) 2.5pc/|{mA\lambda A} 
\ar@{=>}"15";"15"+/va(70) 4pc/_{\alpha m} 
}$$
Finally by naturality this is equal to the diagram
$$
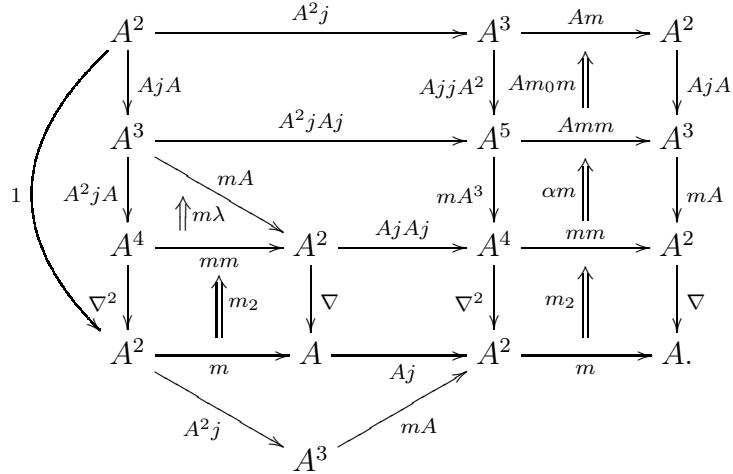
\xymatrix @C4pc {
A^2 \ar[rr]^{A^2j} \ar[d]^{AjA} \ar@/_3pc/[ddd]_{1} &&
A^3 \ar[r]^{Am}_{~}="1" \ar[d]_{AjjA^2} & 
A^2  \ar[d]^{AjA} \\
A^3 \ar[rr]^{A^2jAj} \ar[dr]^{mA}_(.35){~}="10" \ar[d]_{A^2jA} &&
A^5 \ar[r]^{Amm}="2"_{~}="3" \ar[d]_{mA^3} &
A^3 \ar[d]^{mA}  \\
A^4 \ar[r]_{mm}="8"^(.3){~}="9" \ar[d]_{\nabla^2}  & 
A^2 \ar[r]^{AjAj}  \ar[d]^{\nabla} & 
A^4 \ar[r]^{mm}="4"_{~}="5" \ar[d]_{\nabla^2} & 
A^2 \ar[d]^{\nabla} \\
A^2 \ar[r]_{m}^{~}="7"   \ar[dr]_{A^2j} & A \ar[r]_{Aj} &
A^2 \ar[r]_{m}^{~}="6" & A . \\
& A^3 \ar[ur]_{mA} 
\ar@{=>}"2";"1"^{Am_0m} 
\ar@{=>}"4";"3"^{\alpha m} 
\ar@{=>}"6";"5"^{m_2} 
\ar@{=>}"7";"8"_{m_2}
\ar@{=>}"9";"10"_{m\lambda} 
}$$
and now \eqref{eq:2} clearly follows.

This now proves that $(1,\chi)$ defines a morphism 
of skew monoidales from $(A,m,j)$ to $(A,\nabla,j)$.

\begin{theorem}\label{thm:characterization}
The 2-cell $\chi$ defines the unit of a 2-adjunction $R\dashv T$ between 
the 2-category $\Mnd^*(\LBM)$ of opmonoidal monads on lax
braided monoidales, 
and the 2-category $\Skew(\LBM)$ of skew
monoidales in \LBM with unit $j$. The counit $RT\to 1$ is the
isomorphism described above. The image of $T$ consists
of those skew monoidales $(A,m,j)$ for which $\chi$ is 
invertible. 
\end{theorem}

\begin{theorem}\label{thm:Aj-opmonadic}
In the context of the previous theorem, the restriction of $\chi$
along $Aj\colon A\to A^2$ is always invertible, so if restriction along
$Aj$ is conservative then the 2-adjunction $R\dashv T$ is 
in fact an equivalence. In particular this will be the case if 
$Aj$ is opmonadic. 
\end{theorem}

\proof
Use the definition of $\chi$, the fact that $\lambda.j=m_0$, 
and one of the counit laws for the opmonoidal structure on $m$.
\endproof

\section{Quantum categories in the cartesian context}\label{qccc}
  
In this final section we turn to the question of which monoidal bicategories \M have the property that quantum categories in \M are just monads. 

The basic example of such an \M is the bicategory $\Span$
of sets and spans. The cartesian product of sets provides $\Span$ 
with a monoidal structure, although it is not a bicategorical 
product in $\Span$. This bicategory has been studied from 
many points of view; the relevant one here is that it is a
{\em cartesian bicategory} in the sense of \cite{CBII}.

The first property of cartesian bicategories that we use is that 
every left adjoint in a cartesian bicategory is opmonadic, and 
so in particular restriction along any left adjoint is conservative.
Thus the hypotheses of Theorem~\ref{thm:Aj-opmonadic} are
satisfied.

The other key property of a cartesian bicategory \M is that 
every object has a canonical symmetric monoidale structure,
with respect to which every morphism has symmetric opmonoidal 
structure,
and with respect to these, every 2-cell is opmonoidal. It follows
that the forgetful 2-functor 
$\LBM\to\M$ from the 2-category of lax braided monoidales in \M
is a biequivalence.

Combining the previous two theorems we now deduce:

\begin{theorem}
For a (strict) cartesian bicategory \M, the 2-category $\Mnd^*(\M)$
of monads in \M is biequivalent to the 2-category $\Skew(\M)$ of
left skew monoidales in \M with unit $I\to A$ given by the unique
map.
\end{theorem}

The bicategory $\Span$ of spans of sets can be generalized to 
a bicategory $\Span(\CE)$ of spans in a finitely complete 
category $\CE$; taking $\CE$ to be the category of sets and functions, we recover $\Span$ itself.
The bicategory $\Span(\CE)$ is also cartesian, and so
in $\Span(\CE)$ once again quantum categories 
are just monads.

\end{document}